\documentclass{article}
\usepackage{geometry,amsmath,amssymb,bbm}
\usepackage{enumitem}
\usepackage{hyperref}

\newenvironment{proof}{\noindent\textbf{Proof\ }}{\hspace*{\fill}$\Box$\medskip}
\newtheorem{lemma}{Lemma}
\newtheorem{proposition}{Proposition}
\newtheorem{corollary}{Corollary}
\newtheorem{definition}{Definition}
\newtheorem{remark}{Remark}

\newtheorem{theorem}{Theorem}

\begin{document}

\title{Whittaker functions for Steinberg representations of $GL(n)$ over a $p$-adic field} \author{Ehud Moshe Baruch and Markos Karameris}\maketitle

 Let $G=GL_{n}(F)$ and let $(\pi_{St},V)$ be a (generalized) Steinberg representation of $G$. It is well known that the space of Iwahori fixed vectors in $V$ is one dimensional. The Iwahori Hecke algebra acts on this space via a character. We determine the value of this character on a particular Hecke algebra element and use this action to determine in full the Whittaker function associated with an Iwahori fixed vector generalizing a result of Baruch and Purkait for $GL_2(F)$. We show that the Iwahori fixed vector is "new" in the sense that it is not fixed by any larger parahoric. We also show that the restriction of the (generalized) Steinberg representation to $SL_n(F)$ remains irreducible hence we get the Whittaker function attached to a Steinberg representation of $SL_n(F)$.

\section{Introduction and Notations}
In \cite{reeder}, Reeder considers two eigenvectors of the finite part $\mathcal{H}_{\mathbf{W}}$ of the Iwahori Hecke algebra $\mathcal{H}(G,J)$ corresponding to the spherical and the Steinberg representation. The spherical Whittaker function is given by the well known Casselman Shalika formula \cite{shint}. For a split reductive group, Reeder computes the Whittaker function of the non-spherical vector on the diagonal values and points that computing the values for a general element can be difficult. In \cite{Bump}, Bump, Brubaker, Buciumas Henrik and Gustaffson give formulas for Whittaker functions of a basis of elements of Iwahori fixed vectors in an induced representation. This however does not immediately provide a simple formula for the Whittaker function in question. 

Let $F$ be a non-archimedean local field and $|.|$ the standard absolute value on $F$. We denote by $\mathfrak{p}$ the maximal ideal, generated by the uniformizer $\varpi$, $\mathcal{O}$ the ring of integers and $\mathfrak{f}=\mathcal{O}/\mathfrak{p}$ the residue field with $q$ elements. Let $G=GL_n(F)$ and consider the following subgroups inside it: the center $\mathcal{Z}$ of $G$ of scalar matrices $z(t)$ with $t\in F$,
 the maximal compact subgroup $K=GL_n(\mathcal{O})$,
 the standard Borel subgroup $B$ of upper triangular, matrices in $G$,
 the maximal unipotent $N$ of upper triangular unipotent matrices,
 the restriction $N_{\mathcal{O}}=N\cap K$,
 the maximal torus $T$ of diagonal matrices $d(t_1,...,t_n)=diag(t_1,...,t_n)$ with entries $t_i\in F^*$ and also using the notation $\varpi^{\bar{k}}=\varpi^{(k_1,k_2\dots,k_n)}=d(\varpi^{k_1},\dots,\varpi^{k_n})$ we define the set $T^+ = \{\varpi^{\bar{k}}| k_1\geq k_2\geq \dots k_n \}$ and $T^0=T\cap K=T\cap J$,
 the Weyl group of permutation matrices $\mathbf{W}$ generated by the simple reflections $s_1,...,s_{n-1}$ and $w_0$ the longest Weyl element and 
 the Iwahori subgroup $J$ which is the preimage of the Borel in $GL_n(\mathfrak{f})$ under the canonical projection induced by $K\to GL_n(\mathfrak{f})$.

Denote by $\Phi=\{\alpha_{i,j}|i\ne j\}$ the associated root system of type $A_{n-1}$ with positive simple roots $\Delta=\{\alpha_{i,i+1},i=1\dots n-1\}$ where $\alpha_{i,j} : T \to F^{*}$ are given by 
$\alpha_{i,j}(d(t_1,t_2,...,t_n))=t_{i}/t_{j}$. Let $e_{i,j}$ be the $n\times n$ matrix with $1$ in the $(i,j)$ entry and zeroes every where else.
Then $x_{i,j}(t)=I+te_{i,j}$ are the one parameter groups corresponding to the roots of $\Phi$. We also denote these as $x_{\alpha}(t)$ for $\alpha\in\Phi$. Then the elements $d\in T$ act on the roots by conjugation as $dx_{\alpha}(t)d^{-1}=x_{\alpha}(\alpha(d)t)$ and similarly for the Weyl elements: $wx_{\alpha}(t)w^{-1}=x_{w\alpha}(t)$.

\section{The Iwahori-Bruhat decomposition}
Following Roger Howe \cite{rog} set:
$
s_0 =
\begin{pmatrix}	
    0 &  \ldots & \varpi^{-1} \\
		
		\vdots & I_{n-2} & \vdots \\
		\varpi & \dots & 0
	\end{pmatrix} 
$
and
$
u=\begin{pmatrix}
    0 & I_{n-1} \\
   \varpi & 0 
\end{pmatrix}
$

Let $\mathcal{H}(G,J)$ be the Hecke algebra of $J$ bi-invariant and compactly supported functions on $G$. $\mathcal{H}(G,J)$ acts on the space of $J$ invariant vectors in a representation of $G$ and 
in particular on right $J$ invariant functions on $G$. We denote by
$X_g\in \mathcal{H}(G,J)$ the characteristic function of $JgJ$. Then the left action of the Hecke algebra on a $J$ invariant function $F:G\to\mathbbm{C}$ is given by $X_g(F)(x)=\int_{JgJ}F(xh)dh$ (see 1.2 of \cite{bor}). 
In fact we can compute this action explicitly:
\begin{lemma} \label{cosum1}
    For any $g\in G$ with coset decomposition $JgJ=\overset{m}{\underset{i=1}{\bigcup}}\gamma_iJ$ and $f$ a $J$ invariant function. Then $$X_{g}(f)(x)=\sum\limits_{i=1}^mf(x\gamma_i).$$
\end{lemma}

\begin{theorem} [Iwahori, Matsumoto \cite{iwmats} Prop. 3.8] ~\\
$\mathcal{H}(G,J)$ is generated by $X_{s_0}, X_{s_i},i\in\{0,...,n-1\}$ and $ X_u$. The algebraic structure on $\mathcal{H}(G,J)$ is given by the relations:\\
$
1) (X_{s_{i}}-q)(X_{s_{i}}+1)=0\:,\;\;\;i=0,1,...,n-1 \\
2) X_{u}^n=1, \\
3) X_uX_{s_i}=X_{s_{i+1}}X_u \\
4) X_{s_i}X_{s_j}X_{s_i}=X_{s_j}X_{s_i}X_{s_j}, i\equiv \pm1 \mod n \\
5) X_{s_i}X_{s_j}=X_{s_j}X_{s_i}, i\not\equiv \pm1 \mod n 
$

\end{theorem}

We let $\mathcal{H}_{\mathbf{W}}$ be the subalgebra of $\mathcal{H}(G,J)$ associated to the finite Weyl group $\mathbf{W}$. The results on coset decompositions below are partially known from Iwahori-Matsumoto \cite{iwmats}:

\begin{lemma}\label{cosrep}
    For any $w\in\mathbf{W}$ with $\ell(w)=k$: $$JwJ=\underset{(t_{\alpha})\in(\mathcal{O}/\mathfrak{p})^k}{\bigcup}\prod\limits_{\alpha\in\Phi|w^{-1}\alpha\in\Phi^-}x_{\alpha}(t_{\alpha})wJ.$$
\end{lemma}
\begin{proof}
By the Iwahori factorization for $J$: $J=N_{\mathcal{O}}T^0N_{\mathfrak{p}^-}$ and $N_{\mathfrak{p}}^-=\prod\limits_{\beta\in\Phi^-}x_{\beta}(\varpi s_{\beta})$ where $N_{\mathcal{O}}=N\cap J$ and $s_{\beta}\in\mathcal{O}$. But $w^{-1}T^0N_{\mathfrak{p}^-}w\subset J$. It thus remains to check the terms in $N_{\mathcal{O}}=\prod\limits_{\alpha\in\Phi^+}x_{\alpha}(s_{\alpha})$. But $w^{-1}x_{\alpha}(s_{\alpha})w=x_{w^{-1}\alpha}(s_{\alpha})\in J$ if and only if $w^{-1}\alpha\in\Phi^+$ or $w^{-1}\alpha\in\Phi^-$ and $s_{\alpha}\in\mathfrak{p}$. Finally observing that $\#\{\alpha\in\Phi|w^{-1}\alpha\in\Phi^-\}=\ell(w^{-1})=\ell(w)$ completes the proof.
\end{proof}\\
Using Lemma \ref{cosrep} and the Iwasawa decomposition $G=NTK$, we obtain the refined Bruhat decomposition of $G$ relative to the Iwahori subgroup:
 \begin{lemma}[Bruhat-Iwahori decomposition \cite{iwmats}]\label{decomp} 
 There exists a double coset decomposition of $G$ given by:
    $$G=\underset{d\in T}{\bigcup}\underset{w\in\mathbf{W}}{\bigcup}NdwJ$$
 \end{lemma}

 \section{Iwahori spherical vectors of Steinberg Representations}
For a smooth multiplicative character $\tau$ of $F^*$ we define a character $\chi=\tau(det(g))$ for $g\in G$. We let $\delta_B$ be the modular character of $B$ (so $\delta_B(d(t_1,t_2,...,t_n))=|t_1|^{n-1}|t_2|^{n-3}...|t_n|^{1-n}$) and denote by $I(\chi)=Ind_B^G(\delta_B^{1/2}\chi)$ the principal series representation formed by normalized induction from $\chi$.
 \begin{definition}
     The generalized Steinberg representation corresponding to $\chi$ is the unique irreducible quotient of $I(\chi\delta_B^{-1/2})$ which we denote by $St^{\chi}$. If $\chi$ is unramified and trivial on $\mathcal{Z}$ then we call this a Steinberg representation.
 \end{definition}
The Steinberg representation is the generic component of this principal series representation. A character $\zeta: T\to\mathbb{C}^{\times}$ is regular if $\zeta=\zeta^{w}$ implies $w=1$. For $\chi=\tau\circ det$ as above satisfying $\chi^{w}=\chi$ we observe that $\chi\delta_B^{-1}$ is a regular character of $T$ i.e. $\chi\delta_B^{-1}=(\chi\delta_B^{-1})^{w}\iff w=1$. \\
The Jacquet functor $V\to(V)_N$ is an exact functor from the category of smooth representations of $G$ to the category of smooth representations of the maximal torus $T$ and can be computed explicitly using a trivial modification of Lemma 8.1.2 of Casselman \cite{cass}:
\begin{proposition} \label{jacdim}
   For any character $\chi$ as above we have that $(St^{\chi})_{N}\simeq \chi\delta_B$.
\end{proposition}
\begin{definition}
    An irreducible admissible representation $(\pi, V )$ of $G$ is called Iwahori spherical if it has a non-zero vector fixed by the Iwahori subgroup $J$. In this case the space of fixed vectors $V^J$ is a simple module for the Iwahori Hecke algebra $\mathcal{H}(G,J)$.
\end{definition}
Assuming $\tau^n=1$ and $\tau$ unramified forces $\tau(\varpi^k)=\epsilon^k$ for some $n$-th root of unity $\epsilon$. In this case Proposition 2.1 of \cite{cass2} gives us the Casselman basis of $I(\chi\delta_B^{1/2})^J$ which consists of the functions:
$$f_w(g)=\begin{cases}
        \chi(b)\delta_B(b) & \text{if } g=bwj, b\in B,j\in J,   \\
        0 & \text{otherwise} 
    \end{cases}.
$$
We write again the Iwahori factorization of the Iwahori subgroup $J=N_{\mathcal{O}}T^0N^-_{\mathfrak{p}}$. Then we have the following result due to Borel \cite{bor} which appears in more detail in Casselman's notes \cite{cassnotes}:
\begin{proposition} \label{jacq}
    For an admissible irreducible representation $(\pi,V)$ of $G$ the canonical projection $V\to (V)_N$ induces an isomorphism $V^J\to(V)_N^{T^0}$.
\end{proposition}
From this point we assume that $\tau$ is unramified with $\tau^n=1$. By extension $\chi$ is also unramified.
\begin{corollary} \label{stdim}
    By the above proposition and Proposition \ref{jacdim}, and since $\delta_B$ and $\chi$ are unramified and thus $T^0$ invariant, it follows that  $dim((St^{\chi})_N^{T^0})=1$ meaning $dim((St^{\chi})^J)=1$.
\end{corollary}
Since the Hecke Algebra $\mathcal{H}(G,J)$ acts on $St^{\chi}_{J}$ and since this space is one dimensional it follows that the Hecke Algebra acts with a character $\rho_{\chi}:\mathcal{H}(G,J) \to \mathbb{C}$.
It is well known that the generators $X_{s_i}$ act on $(St^{\chi})^J$ with the scalar $-1$. We will give an alternative proof and also obtain the action of the generator $X_u$, that is, we will compute $\rho_{\chi}(X_u)$.
\begin{proposition}(p.23 of \cite{reeder}) \label{fixvec}
    The Iwahori fixed part of the full induced representation $I(\chi\delta_B^{1/2})^J$ contains up to scalar multiplication exactly two eigenvectors of $\mathcal{H}_{\mathbf{W}}$ which are given by the formulas $\phi^-(g)=\sum\limits_{w\in W}(-q)^{-\ell(w)}f_w(g)$ and $\phi^+(g)=\sum\limits_{w\in W}f_w(g)$. Furthermore, $\phi^-$ is an eigenvector of the full Hecke algebra $\mathcal{H}(G,J)$.
\end{proposition}
\begin{proof}
    There are only two complex characters of the finite part of $\mathcal{H}(G,J)$ given by $\rho^-:X_{s_{\alpha}}\to -1, \forall\alpha\in\Delta$ and $\rho^+:X_{s_{\alpha}}\to q,\forall\alpha\in\Delta$. For each character we will find the space of equivariant vectors and show that it is given by the vectors $\phi^-$ and $\phi^+$ defined above. Let $\eta(g)=\sum\limits_{w\in W}\lambda_wf_w(g)$ be any eigenfunction in $I(\chi\delta_B^{1/2})$ corresponding to $\rho^-$ and similarly $\theta(g)=\sum\limits_{w\in W}\lambda'_wf_w(g)$ an eigenfunction corresponding to $\rho^+$. From Lemma \ref{cosrep} and induction: $X_w(\eta)(1)=\rho^-(X_w)\eta^-(1)=(-1)^{\ell(w)}\lambda_1$ and $X_w(\theta)(1)=\rho^+(X_w)\theta(1)=q^{\ell(w)}\lambda'_1$ for any $w\in\bf{W}$. We now observe that for any $w'\in \bf{W}$: $$X_w(f_{w'})(1)=\sum\limits_{(t_{\alpha})\in(\mathcal{O}/\mathfrak{p})^{\ell(w)}}f_{w'}(\prod\limits_{\alpha\in\Phi|w^{-1}\alpha\in\Phi^-}x_{\alpha}(t_{\alpha})w)=q^{\ell(w)}\delta_{w,w'}$$ where $\delta$ is the indicator function. Applying $X_w$ to the above eigenvectors: $$X_w(\eta)(1)=q^{\ell(w)}\lambda_w \text{ 
 and  } X_w(\theta)(1)=q^{\ell(w)}\lambda'_w$$ and thus we obtain the coefficient relations: $$\lambda_w=(-q)^{-\ell(w)}\lambda_1 \text{  and  } \lambda'_w=\lambda_1'$$ in each case respectively which implies $\eta(g)=\lambda_1\phi^-(g)$ and $\theta(g)=\lambda_1'\phi^+(g)$. These function are both $\mathcal{H}_{\mathbf{W}}$ eigenfunctions. Indeed $\phi^+$ corresponds to the $K$ spherical vector in $I(\chi\delta_B^{1/2})$ and is therefore not inside the Steinberg subrepresentation. But by Corollary \ref{stdim} the Steinberg contains an Iwahori fixed vector which has to be $\phi^-$ and is therefore an eigenfunction of the full $\mathcal{H}(G,J)$.
\end{proof} 
\par
In the next Proposition we will obtain the action of the generators of the Hecke Algebra on $\phi^-$. 

\par

\begin{proposition} \label{char}
The Hecke algebra character of the twisted Steinberg $St^{\chi}$ denoted $\rho_{\chi}:\mathcal{H}(G,J)\to\mathbbm{C}$ and corresponding to the unramified character $\chi(\varpi^{\bar{k}})=\tau(det(\varpi^{\bar{k}}))=\epsilon^{\sum\limits_{i=1}^{n}k_i}$ is given by $$\rho_{\chi}(X_{s_i})=-1$$ for all simple reflections and 
$$\rho_{\chi}(X_u)=(-1)^{n-1}\epsilon.$$
\end{proposition}
\begin{proof}
The action of $X_{s_i}$ for each simple reflection $s_i$ is easily computed using the action of $X_{s_i}$ on $f_w$ given in the proof of Proposition \ref{fixvec}. Hence it is sufficient to compute the action of the affine part of $\mathcal{H}(G,J)$, that is $\rho_{\chi}(X_u)$.
We now observe that $u=\varpi^{(0,\dots,0,1)}s_{n-1}...s_1$ and set $u'=s_{n-1}...s_1$ and $d_{u'}=\varpi^{(0,\dots,0,1)}$. From Proposition \ref{fixvec} we have that $T_u(\phi)(g)=\sum\limits_{w\in W}(-q)^{-\ell(w)}T_u(f_w)(g)$. For $g=1$ then, from Lemma \ref{cosum1} $T_u(f_w)(1)=f_w(u)=f_w(d_{u'}u')$. This means that $T_u(f_w)(1)=0$ unless $w=u'$ in which case: $$T_u(f_{u'})(1)=f_{u'}(d_{u'}u')=\delta_B(d_{u'})\chi(d_{u'})=q^{n-1}\epsilon.$$ This readily implies now that:
$$T_u(\phi)(1)=(-q)^{-\ell(u')}q^{n-1}=(-q)^{-(n-1)}q^{n-1}\epsilon=(-1)^{n-1}\epsilon.$$

\end{proof} 

\section{Iwahori fixed Whittaker functions for $GL(n)$}
We compute the Whittaker function associated to a Steinberg representation. Let $\psi$ be an unramified additive character $\psi:F\to\mathbbm{C}^{\times}$ we extend it to a character $\psi:N \rightarrow \mathbbm{C}^{\times}$ by 
$\psi(n)=\sum\psi(n_{i,i+1})$. Let $\mathcal{W}(\pi_{St^{\chi}},\psi)$ be the Whittaker model of $St^{\chi}$ with character $\psi$ then $\mathcal{W}(\pi_{St^{\chi}},\psi)\cong St^{\chi}$. We are interested in $\mathcal{W}(\pi_{St^{\chi}},\psi)^J$ which by Corollary \ref{stdim} has dimension $1$. 
These generalised Steinberg representations correspond to twists of the sign representation of $\mathcal{H}(G,J)$ by a root of unity $\epsilon\in\mathbbm{C}^{\times}$ as shown in Proposition \ref{char}. Let $W\in \mathcal{W}(\pi_{St^{\chi}},\psi)^J$ with $W\not\equiv 0$, then $W(gj)=W(g),\forall g\in G$ and $j\in J$ and $(F*W)(g)=\rho_{\chi}(F)W(g)$ for all $F\in\mathcal{H}(G,J)$, with $\rho_{\chi}$ the character of $\mathcal{H}(G,J)$ as above.
From Lemma \ref{decomp} it is clear that the values of $W$ are completely determined by the elements in $TW$ and specifically if $d=d(t_1,...,t_n)$ then $W(dw)=W(d(|t_1|,...,|t_n|)w)$ since $d(t_1,t_2,...t_n)w$ and $d(|t_1|, ...|t_n|)w$ differ by en element of $J$. This means that we can limit our efforts to computing the value of $W$ solely on $\varpi^{\bar{k}}w$. In what follows we always assume $W$ is as above.\\
Define $\langle\alpha_{i,j},\bar{k}\rangle=k_i-k_j$. The following result is a variant of a result which appears in \cite{Bump}:
\begin{lemma} \label{supp}
    Let $d=\varpi^{(k_1,...,k_{n-1},k_n)}=\varpi^{\bar{k}}$, then $W(dw)=0$ unless $\langle\alpha,\bar{k}\rangle\geq \begin{cases}
        0 & \text{if $w^{-1}\alpha\in \Phi^+$} \\
        -1 & \text{if $w^{-1}\alpha\in \Phi^-$} 
    \end{cases}$ for all $\alpha\in\Delta$. A weight $\bar{k}$ satisfying this condition will be called $w$-dominant.
\end{lemma}
\begin{proof}Assume first that there exist $\alpha \in \Delta$ such that $w^{-1}(\alpha) \in \Phi^{+}$ and $\langle\alpha, \bar{k}\rangle  < 0$ or that there exist  $\alpha \in \Delta$ such that $w^{-1}(\alpha) \in \Phi^{-}$ and $\langle\alpha, \bar{k}\rangle  < -1$. Let $s\in F$ be such that $|s|=q$ and $\psi(x_{\alpha}(s))\ne 1$. Then: $\psi(x_{\alpha}(s))W(dw)=W(x_{\alpha}(s)dw)=W(dd^{-1}x_{\alpha}(s)dw)=W(dww^{-1}x_{\alpha}(\varpi^{-\langle\alpha,\bar{k}\rangle} s)w)= W(dwx_{w^{-1}\alpha}(\varpi^{-\langle\alpha,\bar{k}\rangle} s))$. Since $x_{w^{-1}\alpha}(\varpi^{-\langle\alpha,\bar{k}\rangle} s)\in J$ in each case, we then have that \\ $\psi(x_{\alpha}(s))W(dw)=W(dwx_{w^{-1}\alpha}(\varpi^{-\langle\alpha,\bar{k}\rangle} s))=W(dw)$ and thus it follows that $W(dw)=0$. 
\end{proof}
\begin{remark}
The above condition can be translated as follows: $W(dw)=0$ except when $k_i\geq k_{i+1}$ if $w^{-1}$ has an ascend at the place $i$ or $k_i\geq k_{i+1}-1$ if $w^{-1}$ has an descend at the place $i$. Thus $W(dw)\ne 0$ only if $k_1-g_1\geq k_2-g_2\geq...\geq k_n-g_n$ where $g_i=\#\{j\geq i:w^{-1}(j)>w^{-1}(j+1)\}$. 
\end{remark}

\begin{definition} 
    We denote the semigroup of $w$-dominant diagonal matrices by $T_w^+$ and define the element $d_w=\{\varpi^{\bar{k}}\in T_w^+: k_n=0 \text{ and }\langle\alpha,\bar{k}\rangle= \begin{cases}
        0 & \text{if $w^{-1}\alpha\in \Phi^+$} \\
        -1 & \text{if $w^{-1}\alpha\in \Phi^-$} 
    \end{cases},\forall\alpha\in\Delta\}$.
\end{definition}
The element $d_w$ can be realized as the unique maximal valuation element of $T_w^+/\mathcal{Z}$. It is easy to see that $T_w^+=d_wT^+$, so in particular $T^+=T_1^+\subseteq T_w^+,\forall w\in\mathbf{W}$. 
\begin{corollary}
  For any $d\in T$, Lemma \ref{supp} implies that $W(dw)=0$ unless $d\in T_w^+T^0$.
\end{corollary}

\begin{lemma} \label{permlem}
         For a Weyl element $w$ and $d\in T^+$ the following relation is satisfied:
 $$W(dw)=(-q)^{-\ell(w)}W(d).$$
\end{lemma}
\begin{proof}
We use the relation $X_a*X_b=X_{ab}$ if $\ell(a)+\ell(b)=\ell(ab)$. By induction we obtain: $X_{w}(W)(g)=(-1)^{\ell(w)}W(g)$.
From Lemma \ref{cosrep}: $JwJ=\underset{(t_1,\dots t_{\ell(w)})\in \mathcal({O}_F/\mathfrak{p})^{\ell(w)}}{\bigcup}\prod\limits_{j=1}^{\ell(w)}x_{\alpha_j}(t_{j})wJ$ and since $\prod\limits_{j=1}^{\ell(w)}x_{\alpha_j}(t_j)\in N_{\mathcal{O}}$ and $d\in T^+$ it follows that $d\prod\limits_{j=1}^{\ell(w)}x_{\alpha_j}(t_j)d^{-1}\in N_{\mathcal{O}}$. From Lemma \ref{cosum1} then $(-1)^{\ell(w)}W(d)=\sum\limits_{(t_1,\dots t_{\ell(w)})\in \mathcal({O}_F/\mathfrak{p})^{\ell(w)}}W(d\prod\limits_{j=1}^{\ell(w)}x_{\alpha_j}(t_j)w)=q^{\ell(w)}W(dw)\implies\\ W(dw)=(-q)^{-\ell(w)}W(d)$.
\end{proof}

\begin{theorem}[Diagonal Whittaker values] \label{diagon} ~\\
 Let $W\in \mathcal{W}(\pi,\psi)^{J}$ with $W(1)=1$ be such that 
	$$
	W(zg)=W(g),\;z\in Z, g\in G
	$$
	and that
	$$
	X_{s_i}(W)(g)=-W(g) \text{, for } i=1,\dots,n-1,\;\;\;\;X_u(W)(g)=\epsilon W(g)
	$$
	Then $W$ has the following diagonal values:
$$W(\varpi^{(k_1,k_2,\dots,k_n)})=\begin{cases}
       \epsilon^{\sum\limits_{i=1}^{n}k_i}(-1)^{(n-1)\sum\limits_{i=1}^{n}k_i}\delta_B(d)  & \text{if $\varpi^{(k_1,k_2,\dots,k_n)}\in T^+$} \\
        0 & \text{otherwise} 
    \end{cases}.$$
 \end{theorem}
 \begin{proof}
We note again that $$u=\varpi^{(0,0,...,0,1)}s_{n-1}s_{n-2}...s_1$$ We also have that $s_i^2=1$ and : $$s_i\varpi^{\underset{i+1\text{-th place}}{(0,0,\dots,1,\dots,0)}}s_i=\varpi^{\underset{i\text{-th place}}{(0,0,\dots,1,\dots,0)}},i\in\{1,...,n-1\}.$$ 
This gives us the following $n$ relations:
$$u=\varpi^{(0,0,\dots,0,1)}s_{n-1}s_{n-2}...s_1$$
$$s_{n-1}u=\varpi^{(0,...,1,0)}s_{n-2}...s_1$$
$$\vdots$$
$$s_1s_2...s_{n-1}u=\varpi^{(1,0,...,0,0)}$$
Since $X_{u}(W)(g)=W(gu)$ it follows that $W(gu)=\epsilon W(g)$ for all $g \in G$. Using this relation and setting each time $g=ds_i...s_{n-1}$ with $d=d(\varpi^{k_1},\dots,\varpi^{k_n})$, where $d\in T^+$ i.e. $ k_1\geq k_2\geq ...\geq k_{n-1}\geq k_n$ we obtain the relations (*):
$$\epsilon W(d)=W(d\varpi^{(0,0,\dots,0,1)}s_{n-1}s_{n-2}...s_1)$$
$$\epsilon W(ds_{n-1})=W(d\varpi^{(0,...,0,1,0)}s_{n-2}...s_1)$$
$$\vdots$$
$$\epsilon W(ds_2...s_{n-1})=W(d\varpi^{(0,1,...,0,0)}s_1)$$
Using Lemma \ref{permlem} on both sides of the last relation and noting that $\ell(s_j...s_1)=\ell(s_1...s_j)=j$ we see that:
$$W(\varpi^{(k_1,k_2,...,k_{n}}))=\epsilon(-q)^{(n-3)}W(\varpi^{(k_1,k_2+1,...,k_n})).$$ and iterating we get:
$$W(\varpi^{(k_1,k_2,k_3,...,k_n}))=\epsilon^{k_{2}-k_1}(-q)^{(n-3)(k_1-k_{2})}W(\varpi^{(k_1,k_1,k_3,...,k_n})).$$ 
Next we use the second to last relation and similarly get:
$$W(\varpi^{(k_1,k_1,k_3,k_4,...,k_{n}}))=\epsilon^{k_{3}-k_1}(-q)^{(n-5)(k_1-k_{3})}W(\varpi^{(k_1,k_1,k_1,k_4,...,k_{n}})).$$
Continuing in this manner with the rest of (*) we finally obtain: $$W(d)=\prod\limits_{i=1}^n\epsilon^{k_i-k_1}(-q)^{(n+1-2i)(k_1-k_{i})}W(\varpi^{(k_1,k_1,k_1,\dots,k_1)})=\prod\limits_{i=1}^n\epsilon^{k_i-k_1}(-q)^{(n+1-2i)(k_1-k_{i})}.$$ Where the last equality holds because $\varpi^{(k_1,....k_1)}$ is in $\mathcal{Z}$. The result now follows by observing that $\sum\limits_{i=1}^n(2i-1-n)k_1=0$ and that $q^{-\sum\limits_{i=1}^n(1+n-2i)k_i}=\delta_B(d)$.
 \end{proof}
\begin{lemma} \label{diagcos}
  For any Weyl element $w$: $Jd_wwJ=\underset{\bar{s}=(s_{\alpha})}{\bigcup}(\prod\limits_{\alpha\in \Phi^-}x_{\alpha}(\varpi s_{\alpha}))d_wwJ$, where each $s_{\alpha}$ ranges over a set of representatives of $\mathcal{O}/\mathfrak{p}^{n_{\alpha}}$ in $\mathcal{O}$ for some $n_{\alpha}\in\mathbbm{N}\cup\{0\}$ and each $\alpha\in\Phi^-$. 
\end{lemma}
\begin{proof}
    We write $J=N^-_{\mathfrak{p}}T^0N_{\mathcal{O}}$ as before. For any $n\in N_{\mathcal{O}}$ observe that $n=\prod\limits_{i=1}^rx_{a_i}(s_{a_i})$ for some $r\in\mathbb{N}$ with $a_i\in\Delta^+$ and $s_{a_i}\in \mathcal{O}$ for all indices. Since $w^{-1}d_w^{-1}nd_ww=\prod\limits_{i=1}^r(w^{-1}d_w^{-1}x_{a_i}(s_{a_i})d_ww)$ it suffices to show that $x_{w^{-1}(\beta)}{(\beta(d_w^{-1})s_{\beta})}\in J, \forall \beta\in\Delta^+$ and any $s_{\beta}\in\mathcal{O}$.  However, taking cases, we see that if $\beta\in\Delta^+:w^{-1}\beta\in\Phi^+$ then $\beta(d_w^{-1})\in\mathcal{O}$ and similarly if $w^{-1}\beta\in\Phi^-$ then $\beta(d_w^{-1})\in\mathfrak{p}$. Thus in every case
    $w^{-1}d_w^{-1}x_{\beta}(s_{\beta})d_ww=x_{w^{-1}(\beta)}{(\beta(d_w^{-1})s_{\beta})}\in J$.
\end{proof}
\begin{lemma}\label{aux}
    For every Weyl element $w\in\mathbf{W}$ we have $d_{w_0w}=w_0d_ww_0d_{w_0}z_w$ for some $z_w\in\mathcal{Z}$.
\end{lemma}
\begin{proof}
   We observe that $w^{-1}(i)>w^{-1}(i+1)\iff w^{-1}w_0(n+1-i)>w^{-1}w_0(n+1-(i+1))$ thus we have that if $d_w=(\varpi^{-g_1},...,\varpi^{-g_n})$ then the exponents $g'_i$ of $d_{w_0w}$ which are the number of descents of $(w_0w)^{-1}$  occuring after place $i$ are given by $g'_i=n-i-(g_1-g_{n-i+1})$. This imples $d_{w_0w}=(\varpi^{-(n-1)-g_n+g_1},...,\varpi^{-(n-i)-g_{n-i+1}+g_1},...)$ which is easily seen to be $w_0d_ww_0d_{w_0}z(\varpi^{g_1})$. 
\end{proof}
\begin{theorem}[Whittaker values at every cell $Tw$] \label{main} ~\\
Let as before $d=\varpi^{\bar{k}}$, then under the assumptions of Theorem \ref{diagon}
    $$W(dw)=\begin{cases}
      \epsilon^{\sum\limits_{i=1}^{n}k_i}(-1)^{(n-1)\sum\limits_{i=1}^{n}k_i}\delta_B(d)(-q)^{-\ell(w)}  & \text{if $\bar{k}$ is $w$-dominant} \\
        0 & \text{otherwise} 
    \end{cases}$$
\end{theorem}
\begin{proof}
Since every element of the Iwahori Hecke algerba acts as scalar multiplication, $X_{d_ww}(W)(g)=\lambda_{w}W(g)$ for some $\lambda_w\in\mathbbm{C}$. By Lemma \ref{diagcos} this becomes
$\lambda_{w}W(g)=\sum\limits_{\bar{s}}W(g\prod\limits_{\alpha\in \Phi^-}x_{\alpha}(\varpi s_{\alpha})d_ww)$. We set $g=dd_{w_0}w_0(=dw_0d_{w_0}^{-1})$. To do that we first show that $w_0d_{w_0}^{-1}\prod\limits_{\alpha\in \Delta^-}x_{\alpha}(\varpi s_{\alpha})d_{w_0}w_0=\prod\limits_{\alpha\in \Delta^-}x_{w_0(\alpha)}(\varpi s_{\alpha}\alpha(d_{w_0}^{-1}))\in N_{\mathcal{O}}$ since $|\alpha(d_{w_0}^{-1})|=q^{-1}$ and $\psi(dx_{w_0(\alpha)}(\varpi s_{\alpha}\alpha(d_{w_0}^{-1}))d^{-1})=1$ for any $\alpha\in\Phi^-\backslash\Delta^-$ and $d\in T$. Thus we can conjugate $\prod\limits_{\alpha\in \Phi^-}x_{w_0(\alpha)}(\varpi s_{\alpha}\alpha(d_{w_0}^{-1}))$ with any $d\in T^+$. We now get for any $d\in T^+$:  $\lambda_{w}W(dd_{w_0}w_0)=\mu(Jd_wwJ)W(dd_{w_0}w_0d_ww_0w_0w)$.  We already know the values of $W(dw_0)$ when $d\in T^+$ from Lemma \ref{permlem} (but not in $T_{w_0}^+$ which is bigger). We use these values to compute the value of $\lambda_{w}$. From Lemma \ref{aux} the above relation becomes:
$$\lambda_{w}W(dd_{w_0}w_0)=\mu(Jd_wwJ)W(dd_{w_0w}w_0w),\forall d\in T.$$
Setting $w=w_0$ and comparing scalar terms with Lemma \ref{permlem} for $d=d_{w_0}^{-1}$ we get that $\lambda_{w_0}=(-q)^{\ell(w_0)}\mu(Jd_{w_0}w_0J)W(d_{w_0}^{-1})$. Substituting the value of $\lambda_{w_0}$ and setting $d\to dd_{w_0}^{-1}$ leads to: 
$$W(dw_0)=(-q)^{-\ell(w_0)}W(d_{w_0}^{-1})^{-1}W(dd_{w_0}^{-1})=(-q)^{-\ell(w_0)}\delta_B(d_{w_0})W(dd_{w_0}^{-1}), \forall d\in T$$ Similarly we set $w\to w_0w$ and compute $\lambda_{w_0w}=(-q)^{\ell(w_0)-\ell(w)}\mu(Jd_{ww_0}ww_0J)W(d_{w_0}^{-1}d_w)$. Substituting this value of $\lambda_{w_0w}$ and using the above relation for the $W(dw_0)$ term implies:
$$W(dw)=(-q)^{-\ell(w)}W(d_{w}^{-1})^{-1}W(dd_{w}^{-1})=(-q)^{-\ell(w)}\delta_B(d_w)W(dd_{w}^{-1}), \forall d\in T \; \; (**)$$ and the value $W(dw)$ for any $w\in\mathbf{W}$ and $d\in T$ follows from the diagonal values of Theorem \ref{diagon}.
\end{proof}

\begin{corollary}
    If $W\in\mathcal{W}(\pi_{St^{\chi}},\psi)^J$ and $W(1)=0$ then the relations (*) imply that $W(d)=0,\forall d\in T^+ \iff W(d)= 0, \forall d\in T$. Relation (**) then implies that $W(dw)\equiv 0$ in every cell meaning that $W\equiv 0$. This strengthens a more general non-vanishing result in \cite{reeder}.
\end{corollary}

\section{Whittaker Functions and New-vectors}
From the relation $(F*W)(g)=\rho_{\chi}(F)W(g)$ and the uniqueness in Theorem \ref{main} we conclude that: 
\begin{theorem} \label{StWhit}
    The function described in Theorem \ref{main} $$W(dw)=\begin{cases}
     (\chi\delta_B)(d)(-q)^{-\ell(w)}  & \text{if $\bar{k}$ is $w$-dominant} \\
        0 & \text{otherwise} 
    \end{cases}$$ with $d=\varpi^{\bar{k}}$, is the unique Whittaker function $W_v$ that corresponds to the unique Iwahori spherical vector $v\in (St^{\chi})^J$.
\end{theorem}
We now use this explicit description of the Whittaker function associated to $v\in St^{\chi}$ to show that $v$ is not fixed by any other parahoric subgroup of $G$ for all generalized Steinberg representations. \\
 For $S=\{s_{i_1},...,s_{i_r}\}$ a subset of the simple permutations generating $\mathbf{W}$, let $\mathbf{W}_S$ be the subset of $\mathbf{W}$ generated by the elements of $S$.
\begin{theorem}
    For any parahoric $K_S=J\mathbf{W}_SJ$ with $S\ne\{1\}$ it holds that $(St^{\chi})^{K_S}=\{0\}$.
\end{theorem}
\begin{proof}
    Assume $v\in (St^{\chi})^{K_S}$ is non-zero, then $v\in (St^{\chi})^J$ and thus $W_v$ is of the form we compute in Theorem \ref{main}. Then $v\in (St^{\chi})^{K_S} \implies W_v(gk_S)=W(g),\forall k_S\in K_S$ so picking $k_S=s_i$ for some $s_i\in S$ we see that $W(gs_i)=W(g),\forall g\in G$. Setting $g=d_{s_i}$ yields: $W(d_{s_i}s_i)=W(d_{s_i})=0$ since $d_{s_i}\not\in T^+$. This however contradicts the formula in Theorem \ref{main} according to which $W(d_{s_i}s_i)\ne 0$.
\end{proof}

\section{Whittaker Functions for $SL(n)$}
We will use basic notions from $l$-indistinguishability to describe the Iwahori spherical Whittaker functions of $SL_n(F)$ in the same way. Let $G'$ denote the product of $SL_n(F)$ and $\mathcal{Z}$ in $G$. The representations of $G'$ are trivially in bijection with the representations of $SL_n(F)$ up to twist by a fixed central character of $G'$. 
\begin{lemma}
    $St^{\chi}|_{SL_n(F)}$ remains irreducible as a representation of $SL_n(F)$.  
\end{lemma}
\begin{proof}
    Let $\xi$ be a character of $G$ that is trivial on $G'$. Then $\xi=\xi'\circ\det$ where $\xi'$ is a character of $F^*$ and thus $\xi^w=\xi$. Notice that if $\xi\otimes St^{\chi}\simeq St^{\chi}$ then $(\xi\otimes St^{\chi})_{N}\simeq (St^{\chi})_{N}$. But $\xi\otimes I(\chi)\simeq I(\xi\chi)$ and $\xi\chi\delta_B$ is a regular character since $\xi^w\chi^w=\xi\chi$ and $\delta_B$ is regular. By Proposition \ref{jacdim} we thus obtain that $\xi\chi\delta_B\simeq\chi\delta_B$. Since these are characters though this means $\xi\chi\delta_B=\chi\delta_B \implies \xi\equiv 1$. From Lemma 2.1 d) of \cite{LInd} this means $St^{\chi}$ is irreducible as a representation of $G'$ and thus of $SL_n(F)$.
\end{proof}

\begin{theorem}
    The restriction of the Whittaker function described in Theorem \ref{StWhit} to $SL_n(F)$ (i.e. for $d\in T: det(d)=1$) is the unique up to scalar multiplication Iwahori spherical Whittaker function of $SL_n(F)$ corresponding to $(St^{\chi})^J$.
\end{theorem}

Notice that we get only one Whittaker function for all generalized Steinberg representations $(St^{\chi})^J$. This is because in $SL_n(F)$ all of them restrict to the same irreducible representation since $\chi\otimes St\simeq St^{\chi}$ by definition and $\chi$ is trivial on $G'$. 
In Hecke algebra terms this can be observed by restricting the action of $\mathcal{H}(G,J)$ to the action of $X_{s_{i}}$ only.


\begin{thebibliography}{99}

\bibitem{iwmats}
N. Iwahori, H. Matsumoto, On some Bruhat decomposition and the structure of the Hecke rings of $p$-adic Chevalley groups,
Publications Mathématiques de l’I.H.É.S.,  tome 25 (1965), p. 5-48

\bibitem{cass}
W. Casselman, The Steinberg character as a true character,
in Moore, Calvin C. (ed.), Harmonic analysis on homogeneous spaces (Williams Coll., Williamstown, Mass., 1972), Proc. Sympos. Pure Math., vol. XXVI, Providence, R.I.: American Mathematical Society, 413–417, ISBN 978-0-8218-1426-0

\bibitem{bor}
A. Borel, Admissible representations of a semi-simple group over a local field with vectors fixed under an Iwahori subgroup, Inv. Math., 35, 233- 259 (1976)

\bibitem{cassnotes}
W. Casselman, An introduction to the theory of admissible representations of reductive p-adic groups, Unpublished notes

\bibitem{Bump}
B. Brubaker, V. Buciumas, D. Bump, Henrik P. A. Gustafsson, Colored Vertex Models and Iwahori Whittaker Functions, 2019, \href{https://arxiv.org/abs/1906.04140}{arXiv: 1906.04140 [math.RT]}

\bibitem{reeder}
M. Reeder, $p$-adic Whittaker functions and vector bundles on flag manifolds, Comp. Math. 85, No 1 (1980), p. 9 - 36

\bibitem{shint}
T. Shintani, On an Explicit Formula for Class-$1$ "Whittaker Functions" on $GL_n$ over $p$-adic Fields, Proc. Japan Acad. 52, (1976), p. 180 - 182

\bibitem{rog}
R. Howe, Affine-like Hecke algebras and p-adic representation theory in Iwahori-Hecke
Algebras and their Representation Theory, Lecture Notes in Mathematics 1804 (2002), 27–69.

\bibitem{LInd}
S.S. Gelbart, A. W. Knapp, L-indistinguishability and R Groups for the Special Linear Group, Adv. Math. 43 (1982), p. 101-121 

\bibitem{cass2}
W. Casselman, The unramified principal series of p-adic groups. I. The spherical function, Comp. Math. 40, No 3 (1980), p. 387-406

\end{thebibliography}
\end{document}